\input ./style/arxiv-vmsta.cfg
\documentclass[numbers,compress,v1.0.1]{vmsta}

\volume{4}
\issue{3}
\pubyear{2017}
\firstpage{189}
\lastpage{198}
\doi{10.15559/17-VMSTA82}

\startlocaldefs

\newtheorem{theorem}{Theorem}

\newtheorem{proposition}[theorem]{Proposition}

\theoremstyle{definition}
\newtheorem{remark}[theorem]{Remark}

\def\N{\mathbb{N}}
\def\R{\mathbb{R}}

\def\ex{\mathbb{E}}
\newcommand\cadlag{c\`adl\`ag }

\allowdisplaybreaks
\endlocaldefs

\begin{document}
\begin{frontmatter}

\title{The self-normalized Donsker theorem revisited}

\author{\inits{P.}\fnm{Peter}\snm{Parczewski}}\email{parczewski@math.uni-mannheim.de}
\address{University of Mannheim, Institute of Mathematics A5,6, D-68131~Mannheim,~Germany}

\markboth{P. Parczewski}{The self-normalized Donsker theorem revisited}

\begin{abstract}
We extend the Poincar\'{e}--Borel lemma to a weak approximation of a
Brownian motion via simple functionals of uniform distributions on
n-spheres in the Skorokhod space $D([0,1])$. This approach is used to
simplify the proof of the self-normalized Donsker theorem in Cs\"org\H
{o} et al. (2003). Some notes on spheres with respect to $\ell_p$-norms
are given.
\end{abstract}

\begin{keywords}
\kwd{Poincar\'{e}--Borel lemma}
\kwd{Brownian motion}
\kwd{Donsker theorem}
\kwd{self-normalized sums}
\end{keywords}
\begin{keywords}[2010]
\kwd{60F05}
\kwd{60F17}
\end{keywords}


\accepted{9 August 2017}
\revised{9 August 2017}
\received{18 May 2017}
\publishedonline{18 September 2017}

\end{frontmatter}

\section{Introduction}

Let $\mathcal{S}^{n-1}(d) = \{x \in\R^{n}: \|x\|=d\}$
be the $(n-1)$-sphere with radius $d$, where $\|\cdot\|$ denotes the
Euclidean norm. The uniform measure on the unit sphere $\mathcal
{S}^{n-1} := \mathcal{S}^{n-1}(1)$ can be characterized as
\begin{equation}
\label{eq:UnifSphereMeas} \mu_{S,n} \stackrel{d} {=} \frac{(X_1,\ldots, X_n)}{\|(X_1,\ldots, X_n)\|},
\end{equation}
where $(X_1,\ldots, X_n)$ is a standard $n$-dimensional normal random variable.

The celebrated Poincar\'{e}--Borel lemma is the classical result on the
approximation of a Gaussian distribution by projections of the uniform
measure on $\mathcal{S}^{n-1}(\sqrt{n})$ as $n$ tends to infinity: Let
$n \geq m$ and $\pi_{n,m}: \R^n \rightarrow\R^m$ be the natural
projection. The uniform measure on the sphere $\mathcal{S}^{n-1}(\sqrt
{n})$ is given by $\sqrt{n}\,\mu_{S,n}$. Then, for every fixed $m \in\N$,
\[
\sqrt{n}\mu_{S,n} \circ\pi_{n,m}^{-1}
\]
converges in distribution to a standard $m$-dimensional normal
distribution as $n$ tends to infinity, cf. \cite[Proposition~6.1]{Lifshits}. Following the historical notes in \cite[Section~6]{Diaconis} on the earliest reference to this result by \'Emile Borel,
we acquire the usual practice to speak about the Poincar\'{e}--Borel lemma.

Among other fields, this convergence stimulated the development of the
infinite-dimensional functional analysis (cf. \cite{McKean}) as well as
the concentration of measure theory (cf. \cite[Section~1.1]{LedouxTalagrand}).

In particular, it inspired to consider connections of the Wiener
measure and the uniform measure on an infinite-dimensional sphere \cite{Wiener}. Such a Donsker-type result is firstly proved in \cite{Cutland_Ng} by nonstandard methods. For the illustration, we make use
of the notations in \cite{Dryden}, where this result is used for
statistical analysis of measures on high-dimensional unit spheres.
Define the functional
\[
Q_{n,2}:\mathcal{S}^{n-1} \rightarrow C\bigl([0,1]\bigr), \qquad (x_1,\ldots, x_n) \mapsto
\xch{\bigl(Q_{n,2}(t) \bigr)_{t \in[0,1]},}{\bigl(Q_{n,2}(t) \bigr)_{t \in[0,1]}}
\]
such that
\[
Q_{n,2}(k/n) := \xch{\frac{\sum_{i=1}^{k} x_i}{\|(x_1,\ldots, x_n)\|},}{\frac{\sum_{i=1}^{k} x_i}{\|(x_1,\ldots, x_n)\|}}
\]
for $k \in\{0,\ldots, n\}$ and is linearly interpolated elsewhere.
Then \cite[Theorem~2.4]{Cutland_Ng} gives that the sequence of processes
\[
\mu_{S,n} \circ Q_{n,2}^{-1}
\]
converges weakly to a Brownian motion $W := (W_t)_{t \in[0,1]}$ in the
space of continuous functions $C([0,1])$ as $n$ tends to infinity. The
first proof without nonstandard methods in $C([0,1])$ and in the
Skorohod space $D([0,1])$ is given in \cite{Rackauskas}.

In this note, we present a very simple proof of the \cadlag version of
this Poincar\'{e}--Borel lemma for Brownian motion. This is the content
of Section~\ref{section1}.

Some remarks on such Donsker-type convergence results on spheres with
respect to $\ell_p$-norms are collected in Section~\ref{section2}.

In fact, our simple approach can be used to simplify the proof of the
main result in \cite{Csorgo} as well. This is presented in Section~\ref{section3}.

\section{Poincar\'{e}--Borel lemma for Brownian motion}\label{section1}

Suppose $X_1, X_2, \ldots$ is a sequence of i.i.d. standard normal
random variables. Then $(X_1,\ldots, X_n)$ has a standard
$n$-dimensional normal distribution. We define the processes with
\cadlag paths
\[
Z^n = \biggl(Z^n_t := \frac{\sum_{i=1}^{\lfloor nt \rfloor}X_i}{\|
(X_1,\ldots, X_n)\|}
\biggr)_{t \in[0,1]}.
\]
Thus, $Z^n$ is equivalent to $\mu_{S,n} \circ\overline{Q}_{n,2}^{-1}$
for the functional
\[
\overline{Q}_{n,2}: \mathcal{S}^{n-1} \rightarrow D\bigl([0,1]\bigr),
\qquad
(x_1,\ldots, x_n) \mapsto\! \xch{\biggl(\overline{Q}_{n,2}(t) = \frac{\sum_{i=1}^{\lfloor nt \rfloor} x_i}{\|(x_1,\ldots, x_n)\|} \biggr)_{t \in[0,1]},}%
{\biggl(\overline{Q}_{n,2}(t) = \frac{\sum_{i=1}^{\lfloor nt \rfloor} x_i}{\|(x_1,\ldots, x_n)\|} \biggr)_{t \in[0,1]}}
\]
and therefore it is a relatively simple computation from the uniform
distribution on the $n$-sphere. Then the following extension of the
Poincar\'{e}--Borel lemma is true:
\begin{theorem}\label{thm:Poincare_Bm}
The sequence $(Z^n)_{n \in\N}$ converges weakly in the Skorokhod space\break
$D([0,1])$ to a standard Brownian motion $W$ as $n$ tends to infinity.
\end{theorem}

\begin{proof}
As the distribution of the random vector in \eqref{eq:UnifSphereMeas}
is exactly the
uniform measure $\mu_{S,n}$, the proof of the convergence of
finite-dimensional distributions is in line with the classical Poincar\'
{e}--Borel lemma: by the law of large numbers, $\frac{1}{n}\sum_{i=1}^{n} X_i^2 \rightarrow1$ in probability. Hence, by the
continuous mapping theorem, $\sqrt{n}/(\|(X_1,\ldots, X_n)\|)
\rightarrow1$ in probability, and, by Donsker's theorem and Slutsky's
theorem, we conclude the convergence of finite-dimensional distributions.

For the tightness we consider the increments of the process $Z^n$ and
make use of a standard criterion. For all $s\leq t$ in $[0,1]$, we denote
\begin{align}
\label{eq:ThmPoincareBm0} \bigl(Z^n_t - Z^n_s
\bigr)^2 &= \frac{\sum_{ \lfloor ns \rfloor< i \leq\lfloor nt \rfloor
}X_i^2}{\sum_{i\leq n} X_i^2} + \frac{\sum_{ \lfloor ns \rfloor< i
\neq j \leq\lfloor nt \rfloor}X_iX_j}{\sum_{i\leq n} X_i^2} =:
I^{t,s}_1 + I^{t,s}_2.
\end{align}
Due to the symmetry of the standard $n$-dimensional normally
distributed vector\break $(X_1,\ldots, X_n)$, for all pairwise different
$i,j,k,l$, we observe
\begin{align}
\label{eq:ThmPoincareBm1} \ex \biggl[\frac{X_iX_jX_k X_l}{(\sum_{i\leq n} X_i^2)^2} \biggr] = \ex \biggl[
\frac{X_i^2X_jX_k}{(\sum_{i\leq n} X_i^2)^2} \biggr] =0.
\end{align}
Let $s \leq u \leq t$ in $[0,1]$. Thus via \eqref{eq:ThmPoincareBm1},
we conclude
\[
\ex \bigl[ I^{t,u}_1 I^{u,s}_2 \bigr]= 0,
\qquad \ex \bigl[ I^{t,u}_2 I^{u,s}_1\bigr] = 0,
\qquad \ex \bigl[ I^{t,u}_2 I^{u,s}_2
\bigr] = 0,
\]
and therefore
\[
\ex \bigl[ \bigl(Z^n_t - Z^n_u
\bigr)^2 \bigl(Z^n_u - Z^n_s
\bigr)^2 \bigr] = \ex \bigl[ I^{t,u}_1
I^{u,s}_1 \bigr].
\]
We denote for shorthand $m_1:= \lfloor nt \rfloor-\lfloor nu \rfloor$,
$m_2 := \lfloor nu \rfloor-\lfloor ns \rfloor$ and $m_3 := n-(\lfloor
nt \rfloor-\lfloor ns \rfloor)$. Then we observe
\[
I^{t,u}_1 I^{u,s}_1 =
\frac{\chi^2_{m_1} \chi^2_{m_2}}{ (\chi^2_{m_1} + \chi^2_{m_2} +
\chi^2_{m_3} )^2} =
\xch{\frac{\frac{1}{2} ((\chi^2_{m_1}+\chi^2_{m_2})^2 - (\chi^2_{m_1})^2 - (\chi^2_{m_2})^2 )}{ (\chi^2_{m_1} + \chi^2_{m_2} + \chi^2_{m_3} )^2},}%
{\frac{\frac{1}{2} ((\chi^2_{m_1}+\chi^2_{m_2})^2 - (\chi^2_{m_1})^2 - (\chi^2_{m_2})^2 )}{ (\chi^2_{m_1} + \chi^2_{m_2} + \chi^2_{m_3} )^2}}
\]
for pairwise independent chi-squared random variables $\chi^2_m$ with
$m$ degrees of freedom. We recall that
$\frac{\chi^2_{m} }{\chi^2_{m} + \chi^2_{k}}$ is $\text{Beta}(m/2,
k/2)$-distributed with
\begin{align}
\label{eq:ThmPoincareBm2} \ex \biggl[ \biggl(\frac{\chi^2_{m} }{\chi^2_{m} + \chi^2_{k}} \biggr)^2 \biggr]
= \biggl(\frac{m+2}{m+k+2} \biggr) \biggl(\frac{m}{m+k} \biggr).
\end{align}
Hence a computation via \eqref{eq:ThmPoincareBm2} yields
\begin{align*}
\ex \bigl[I^{t,u}_1 I^{u,s}_1 \bigr]
&= \frac{m_1 m_2}{(m_1+m_2+m_3+2)(m_1+m_2+m_3)}
\\
&\leq \biggl(\frac{m_1}{m_1+m_2+m_3} \biggr)
\xch{\biggl(\frac{m_2}{m_1+m_2+m_3} \biggr),}%
{\biggl(\frac{m_2}{m_1+m_2+m_3} \biggr)}
\end{align*}
and therefore
\begin{align*}
\ex \bigl[\bigl(Z^n_t - Z^n_u
\bigr)^2 \bigl(Z^n_u - Z^n_s
\bigr)^2 \bigr] &\leq \biggl(\frac
{\lfloor nt \rfloor-\lfloor nu \rfloor}{n} \biggr) \biggl(
\frac{\lfloor
nu \rfloor-\lfloor ns \rfloor}{n} \biggr)
\\
&\leq \biggl(\frac{\lfloor nt \rfloor-\lfloor ns \rfloor}{n} \biggr)^2.
\end{align*}

Thus the well-known criterion \cite[Theorem~15.6]{Billi_alt} (cp.
Remark~1 in \cite{P}) implies the tightness of $Z^n$.
\end{proof}

\begin{remark}

(i) The heuristic connection of the Wiener measure and the uniform measure
on an infinite-dimensional sphere goes back to Norbert Wiener's study of the
\emph{differential space}, \cite{Wiener}. The first informal presentation of
Theorem~\ref{thm:Poincare_Bm} and further historical notes can be found in
\cite{McKean}. The first rigorous proof is given in Section~2 of
\cite{Cutland_Ng}. However, the authors make use of nonstandard analysis and
the functional $Q_{n,2}$. To the best of our knowledge, the first proof of
Theorem~\ref{thm:Poincare_Bm} is \cite{Rackauskas}. In contrast, our proof is
based on the pretty decoupling in the tightness argument. Moreover, this
approach is extended in Section~\ref{section3} to a simpler proof of Theorem~1 in \cite{Csorgo}.

(ii) According to the historical comments in \cite[Section~2.2]{Vershik}, the Poincar\'{e}--Borel lemma could be also attributed
to Maxwell and Mehler.
\end{remark}

\section{$\ell^n_p$-spheres}\label{section2}

In this section, we consider uniform measures on $\ell^n_p$-spheres and prove
that the limit in Theorem~\ref{thm:Poincare_Bm} is the only case
such that a simple $\overline{Q}$-type pathwise functional leads to a
nontrivial limit (Theorem~\ref{thm:PoincareSpheres}).

Furthermore, we present random variables living on $\ell^n_p$-spheres, with a
similar characterization for a fractional Brownian motion
(Theorem~\ref{thm:DonskerSelfNormalizedFbm}).

Concerning the $\ell^n_p$ norm $\|x\|_p =  (\sum_{i=1}^{n}|x_i|^p )^{1/p}$
for $p \in[1,\infty)$ and defining the $\ell^n_p$ unit sphere
\[
\mathcal{S}^{n-1}_p := \bigl\{x \in\R^n : \|x
\|_p=1\bigr\},
\]
the uniform measure $\mu_{S,n,p}$ on $\mathcal{S}^{n-1}_p$ is
characterized similarly to the uniform measure on the Euclidean unit
sphere by independent results in \cite[Lemma~1]{Schechtman} and \cite
[Lemma~3.1]{Rachev}:
\begin{proposition}\label{prop:uniformEllPSphere}
Suppose $X,X_1, X_2,\ldots$ is a sequence of i.i.d. random variables
with density
\[
f(x) = \frac{\exp(-|x|^p/p)}{2p^{1/p}\varGamma(1+1/p)}.
\]
Then
\[
\mu_{S,n,p} \stackrel{d} {=} 
\frac{(X_1,\ldots, X_n)}{\|(X_1,\ldots, X_n)\|_p}.
\]
\end{proposition}

\begin{remark}\label{remark:PoincareEllp}

(i) We notice that the uniform measure on the $\ell^n_p$-sphere
equals the surface measure only in the cases $p \in\{1,2,\infty\}$,
see e.g. \cite[Section~3]{Rachev} or the interesting study of the total
variation distance of these measures for $p \geq1$ in \cite{Naor}.

(ii) In particular, we have a counterpart of the classical Poincar\'
{e}--Borel lemma for finite-dimensional distributions: For every fixed
$m \in\N$,
\[
n^{1/p}\mu_{S,n,p} \circ\pi_{n,m}^{-1}
\]
converges in distribution to the random vector $(X_1,\ldots, X_m)$ as $n$
tends to infinity. This follows immediately from $\ex[|X|^p]=1$ and the law
of large numbers, cf. \cite[Proposition~6.1]{Lifshits} or the
finite-dimensional convergence in Theorem~\ref{thm:Poincare_Bm}.
\end{remark}

Similarly to the characterization of the central limit theorem, cp.
\cite[Theo-\break{}rem~4.23]{Kallenberg}, but in contrast to the convergence of the
projection on a finite number of coordinates in Remark~\ref
{remark:PoincareEllp}, we have a uniqueness result for the processes
constructed according to the $\overline{Q}$-type pathwise functionals.

In the following we denote the convergence in distribution by $\stackrel
{d}{\rightarrow}$ and the almost sure convergence by $\stackrel
{a.s.}{\rightarrow}$.

\begin{theorem}\label{thm:PoincareSpheres}
Suppose $p \geq1$ and denote
\[
\overline{Q}_{n,p}: (x_1,\ldots, x_n) \mapsto
\biggl(\frac{\sum_{i=1}^{\lfloor nt \rfloor} x_i}{\|(x_1,\ldots, x_n)\|_p} \biggr)_{t \in[0,1]}.
\]
Then, in the Skorokhod space $D([0,1])$, as $n$ tends to infinity:
\[
\mu_{S,n,p} \circ\overline{Q}_{n,p}^{-1}
\begin{cases} \stackrel{a.s.}{\rightarrow} 0, &p <2,\\
\stackrel{d}{\rightarrow} W, &p=2,\\
\textnormal{is divergent}, &p>2.
\end{cases} %
\]
\end{theorem}

\begin{proof}
The strong law of large numbers \cite[Theorem~4.23]{Kallenberg} implies that\break
$n^{1/p}/\|(X_1,\ldots, X_n)\|_p \rightarrow1$ almost surely for all
$p\geq1$. Moreover, for $p<2$, it gives as well that $\frac {1}{n^{1/p}}
\sum_{i=1}^{\lfloor nt \rfloor} X_i \rightarrow0$ almost surely for all $t
\in[0,1]$. Thanks to Proposition~\ref {prop:uniformEllPSphere}, we have
\[
\mu_{S,n,p} \circ\overline{Q}_{n,p}^{-1} \stackrel{d}
{=} \frac
{n^{1/p}}{\|(X_1,\ldots, X_n)\|_p} \Biggl(n^{-1/p} \sum_{i=1}^{\lfloor n
\cdot\rfloor}
X_i \Biggr).
\]
Thus we conclude via $n^{-1/p}= n^{-1/2} n^{(p-2)/2p}$, Donsker's
theorem and Slutsky's theorem.
\end{proof}

However, the $\ell^n_p$ spheres can be involved in another convergence
result. The fractional Brownian motion $B^H=(B^H_t)_{t \geq0}$ with
Hurst parameter $H \in(0,1)$ is a centered Gaussian process with the
covariance $\ex[B^H_t B^H_s] = \frac{1}{2}(t^{2H}+s^{2H}-|t-s|^{2H})$.
We refer to \cite{Mishura} for further information on this
generalization of the Brownian motion beyond semimartingales. In
particular, there is the following random walk approximation (\cite[Theorem~2.1]{Taqqu} or \cite[Lemma~1.15.9]{Mishura}): Let $\{X_i\}_{i
\geq1}$ be a stationary Gaussian sequence with $\ex[X_i]=0$ and correlations
\[
\sum_{i,j =1}^{n} \ex[X_i
X_j] \sim n^{2H} L(n),
\]
as $n$ tends to infinity for a slowly varying function $L$. Then $\frac
{1}{n^{2H} L(n)}\sum_{i=1}^{\lfloor nt \rfloor}X_i$ converges weakly in
the Skorohod space $D([0,1])$ towards a fractional Brownian motion with
Hurst parameter $H$. For simplification let $X_i = B^H_i - B^H_{i-1}$,
$i \in\N$, be the correlated increments of the fractional Brownian
motion $B^H$. The stationarity and the ergodic theorem imply, for $p>0$
and the constant $c_H:= \ex[|B^H_1|^{1/H}]$, that
\begin{align}
\label{eq:ErgodicThmFBm}
\bigl(\bigl\|(X_1,\ldots, X_n)\bigr\|_p/n^{H} \bigr)^p= n^{-Hp} \sum
_{i=1}^{n} |X_i|^p
\stackrel{a.s.} {\rightarrow} %
\begin{cases} 0, & p>1/H,\\ c_H &p=1/H,\\ +\infty, & p<1/H,
\end{cases} %
\end{align}
(see e.g. \cite[Eq. (1.18.3)]{Mishura}). With this at hand, we obtain a
similar uniqueness result:

\begin{theorem}\label{thm:DonskerSelfNormalizedFbm}
Let $X_i = B^H_i - B^H_{i-1}$, $i \in\N$, be the increments of a
fractional Brownian motion $B^H$. Then, in the Skorokhod space
$D([0,1])$, as $n$ tends to infinity:
\[
\overline{Q}_{n,p}(X_1,\ldots, X_n) = \biggl(
\frac{\sum_{i=1}^{\lfloor
nt \rfloor} X_i}{\|(X_1,\ldots, X_n)\|_p} \biggr)_{t \in[0,1]} %
\begin{cases} \stackrel{a.s.}{\rightarrow} 0, &p <1/H,\\
\stackrel{d}{\rightarrow} B^H/c_H^H, &p=1/H,\\
\textnormal{is divergent}, &p>1/H.
\end{cases} %
\]
\end{theorem}

\begin{proof}
Taqqu's limit theorem implies, for all $H \in(0,1)$,
\[
\Biggl(n^{-H}\sum_{i=1}^{\lfloor nt \rfloor}
X_i \Biggr)_{t \in[0,1]} \stackrel{d} {\rightarrow}
B^H
\]
in the Skorokhod space $D([0,1])$. Then, thanks to \eqref {eq:ErgodicThmFBm},
we conclude as in Theorem~\ref{thm:PoincareSpheres}.
\end{proof}

\begin{remark}
Due to the different correlations between the random variables $X_i$ in
Theorem~\ref{thm:DonskerSelfNormalizedFbm}, there is no symmetric and trivial
sequence of measures $\hat{\mu}_{S,n,p}$ on the $\ell ^n_p$-spheres and some
simple $\overline{Q}_{n,p}$-type pathwise functionals, which represent the
distributions of $\overline {Q}_{n,p}(X_1,\ldots, X_n)$. However, it would be
interesting, whether some uniform or surface measures on geometric objects in
combination with simple $\overline{Q}_{p}$-type pathwise functionals allow
similar Donsker-type theorems for fractional Brownian motion or other
Gaussian processes?
\end{remark}

\section{The self-normalized Donsker theorem}\label{section3}

Suppose $X,X_1,X_2,\ldots$ is a sequence of i.i.d. nondegenerate random
variables and we denote for all $n \in\N$,
\[
S_n := \sum_{i=1}^{n}X_i,
\qquad V_n^2 := \sum_{i=1}^n
X_i^2.
\]
Limit theorems for self-normalized sums $S_n/V_n$ play an important
role in statistics, see e.g. \cite{Gine}, and have been extensively
studied during the last decades, cf. the monograph on self-normalizes
processes \cite{delaPena}.

In \cite{Csorgo}, the following invariance principle for
self-normalized sum processes is established.

\begin{theorem}[Theorem~1 in \cite{Csorgo}]\label{thm:Csorgo}
Assume the notations above and denote
\[
Z^n_t := S_{\lfloor nt \rfloor}/V_n.
\]
Then the following assertions, with $n$ tending to infinity, are equivalent:
\begin{enumerate}
\item[(a)] $E[X]=0$ and $X$ is in the domain of attraction of the
normal law (i.e. there exists a sequence $(b_n)_{n \geq1}$ with
$S_n/b_n \stackrel{d}{\rightarrow} \mathcal{N}(0,1)$).
\item[(b)] For all $t_0 \in(0,1]$, $Z^n_{t_0} \stackrel{d}{\rightarrow
} \mathcal{N}(0,t_0)$.
\item[(c)] $(Z^n_t)_{t \in[0,1]}$ converges weakly to $(W_t)_{t \in
[0,1]}$ on $(D([0,1]), \rho)$, where $\rho$ denotes the uniform topology.
\item[(d)] On an appropriate joint probability space, the following is valid:
\[
\sup_{t \in[0,1]} \bigl|Z^n_{t} - W(nt)/\sqrt{n}\bigr| =
o_P(1).
\]
\end{enumerate}
\end{theorem}

\begin{remark}
The equivalence of $(a)$ and $(b)$ is the celebrated result \cite
[Theorem~3]{Gine}. Since the implications $(d) \Rightarrow(c)
\Rightarrow(b)$ are trivial, the proof in \cite{Csorgo} is completed
by showing $(a) \Rightarrow(d)$.

Thanks to a tightness argument as in the proof of Theorem~\ref
{thm:Poincare_Bm}, we obtain a simpler alternative for the proof.
\end{remark}

\begin{proof}[Proof of Theorem~\ref{thm:Csorgo}]
As stated in the remark, we already know that $(d) \Rightarrow(c)
\Rightarrow(b) \Leftrightarrow(a)$. We denote
\begin{enumerate}
\item[($c_0$)] $(Z^n_t)_{t \in[0,1]}$ converges weakly to $(W_t)_{t
\in[0,1]}$ on the Skorokhod space $D([0,1])$.
\end{enumerate}
By the continuity of the paths of the Brownian motion and \cite[Section~18]{Billi_alt}, we obtain the equivalence $(c) \Leftrightarrow(c_0)$.
We denote by $d_0$ the Skorokhod metric on $D([0,1])$ which makes it a
Polish space.
The Skorokhod--Dudley Theorem \cite[Theorem~4.30]{Kallenberg} and
$(c_0)$ imply
\[
d_0 \bigl(\bigl(Z^n_t\bigr)_{t \in[0,1]},
(W_t)_{t \in[0,1]} \bigr) \rightarrow \xch{0,}{0}
\]
almost surely on an appropriate probability space. Since the uniform
topology is finer than the Skorokhod topology (\cite[Section~18]{Billi_alt}), we conclude assertion $(d)$. Thus it remains to prove
$(a) \Rightarrow(c_0)$.
Firstly we consider finite-dimensional distributions. Due to \cite
[Lemma~3.2]{Gine}, the sequence $(b_n)_{n \in\N}$ with
$S_n/b_n \stackrel{d}{\rightarrow} \mathcal{N}(0,1)$ fulfills $V_n/b_n
\rightarrow1$ in probability and $b_n = \sqrt{n} L(n)$ for some slowly
varying at infinity function $L$. The continuous mapping theorem
implies $b_n/V_n \rightarrow1$ in probability. Take arbitrary $N \in\N
$, $a_1,\ldots, a_N \in\R$ and $t_1,\ldots, t_N \in[0,1]$. Without
loss of generality, we assume $t_1< \cdots< t_N$ and denote $t_0:=0$
and $t_{N+1}:=1$. Then, by the independence of the random variables
$S_{\lfloor n t_i \rfloor} - S_{\lfloor n t_{i-1} \rfloor}$,
$i =1, \ldots, N+1$,
for every fixed $n \in\N$, L\'evy's continuity theorem
and the normality of the random vector $(Y_1, \ldots, Y_{N+1})$, we obtain
\[
\biggl(\frac{S_{\lfloor n t_1 \rfloor} - S_{\lfloor n t_{0} \rfloor
}}{\sqrt{(\lfloor n t_1 \rfloor)}}, \ldots, \frac{S_{\lfloor n t_{N+1}
\rfloor} - S_{\lfloor n t_{N} \rfloor}}{\sqrt{(\lfloor n t_{N+1}
\rfloor- \lfloor n t_{N} \rfloor)}} \biggr) \stackrel{d} {
\rightarrow} (Y_1, \ldots, Y_{N+1}),
\]
as $n$ tends to infinity. As the sequence $(b_n)_{n \in\N}$ is
regularly varying with exponent $1/2$, it is easily seen that
\[
\frac{b_{\lfloor n t_i \rfloor- \lfloor n t_{i-1} \rfloor}}{b_n} \rightarrow\sqrt{t_i - t_{i-1}}.
\]
Via the continuous mapping theorem, we conclude
\begin{align*}
\sum_i a_i \frac{S_{\lfloor n t_i \rfloor}}{b_{n}} &=
\sum_{i=1}^{N+1} \frac{ (\sum_{j \leq i} a_j ) (b_{\lfloor n t_i \rfloor-
\lfloor n t_{i-1} \rfloor})}{b_n} \biggl(
\frac{S_{\lfloor n t_i \rfloor}
- S_{\lfloor n t_{i-1} \rfloor}}{b_{\lfloor n t_i \rfloor- \lfloor n
t_{i-1} \rfloor}} \biggr)
\\
&\stackrel{d} {\rightarrow} \sum_{i=1}^{N+1}
\biggl(\sum_{j \leq i} a_j \biggr)
\sqrt{t_i - t_{i-1}} Y_i \stackrel{d} {=} \sum
_{i=1}^{N+1} a_i
W_{t_i}.
\end{align*}
Slutsky's theorem implies
\[
\sum_{i=1}^{N+1} a_i
Z^n_{t_i} = \biggl(\frac{b_n}{V_n} \biggr) \biggl(\sum
_i a_i \frac{S_{\lfloor n t_i \rfloor}}{b_{n}} \biggr)
\stackrel{d} {\rightarrow} \sum_{i=1}^{N+1}
a_i W_{t_i},
\]
what means the convergence of finite-dimensional distributions.

The tightness follows again by the criterion \cite[Theorem~15.6]{Billi_alt}. By the identical distribution, for all $m \leq n$, we have
\begin{align}
\label{eq:ThmCsorgo1} \ex \biggl[ \biggl(\frac{\sum_{i \leq m} X_i^2}{\sum_{i\leq n} X_i^2} \biggr)^2
\biggr] = \ex \biggl[\frac{m X_1^4}{(\sum_{i\leq n} X_i^2)^2} \biggr] + \ex \biggl[\frac{m(m-1) X_1^2 X_2^2}{(\sum_{i\leq n} X_i^2)^2}
\biggr].
\end{align}
Thanks to the value $1$ on the left hand side in \eqref{eq:ThmCsorgo1}
for $m=n$, we conclude
\begin{align*}
0 \leq\ex \biggl[\frac{X_1^2 X_2^2}{(\sum_{i\leq n} X_i^2)^2} \biggr] \leq\frac{1}{n(n-1)}.
\end{align*}
In contrast to \eqref{eq:ThmPoincareBm1}, for possibly nonsymmetric
random variables, the Cauchy--Schwarz inequality and \cite
[(3.10)]{Gine} yields a constant $c_X<\infty$ such that for every $r\in
\{2,3,4\}$,
\begin{align}
\label{eq:ThmCsorgo2}
\max_{\substack{i,j,k,l \leq n\\ |\{i,j,k,l\}| =r}} \ex \biggl[\frac
{|X_iX_jX_k X_l|}{(\sum_{i\leq n} X_i^2)^2} \biggr]
\leq c_X n^{-r}.
\end{align}
Applying the estimates in \eqref{eq:ThmCsorgo2} on the terms in \eqref
{eq:ThmPoincareBm0} gives that
\begin{align}
\label{eq:ThmCsorgo3} \max_{i,j \in\{1,2\}} \ex \bigl[ I^{t,u}_i
I^{u,s}_j \bigr] \leq c_X \biggl(
\frac{\lfloor nt \rfloor-\lfloor ns \rfloor}{n} \biggr)^2.
\end{align}
Hence, we obtain
\[
\ex \bigl[\bigl(Z^n_t - Z^n_u
\bigr)^2 \bigl(Z^n_u - Z^n_s
\bigr)^2 \bigr] = \ex\bigl[\bigl(I^{t,u}_1 +
I^{t,u}_2\bigr) \bigl(I^{u,s}_1 +
I^{u,s}_2\bigr)\bigr]\leq4 c_X
\xch{\biggl(\frac{\lfloor nt \rfloor-\lfloor ns \rfloor}{n} \biggr)^2,}%
{\biggl(\frac{\lfloor nt \rfloor-\lfloor ns \rfloor}{n} \biggr)^2}
\]
and the proof concludes as in Theorem~\ref{thm:Poincare_Bm}.
\end{proof}

\begin{remark}
(i) By the same reasoning, we obtain Theorem~\ref {thm:PoincareSpheres} for
the sequence of i.i.d. variables $X, X_1, X_2, \ldots$ such that
Theorem~\ref{thm:Csorgo} $(a)$ is fulfilled.

(ii) In \cite{Csorgo2}, a similar counterpart of Theorem~\ref {thm:Csorgo}
for $\alpha$-stable L\'evy processes is established. An interesting question
would be on a uniqueness result similar to Theorem~\ref{thm:PoincareSpheres}.
\end{remark}

\section*{Acknowledgments}
The author thanks the anonymous referees for careful reading and
helpful comments.

\end{document}